\documentclass[12pt,reqno]{amsart}

\usepackage{amsmath,graphicx}

\newcommand{\R}{\mathbb{R}}

\theoremstyle{plain} \newtheorem{theorem}{Theorem}[section]
\newtheorem{lemma}[theorem]{Lemma}
\newtheorem{cor}[theorem]{Corollary}
\newtheorem{prop}[theorem]{Proposition}
\newtheorem{conjecture}[theorem]{Conjecture}

\theoremstyle{definition}
\newtheorem{assumption}[theorem]{Assumption}

\DeclareMathOperator{\supp}{supp} \DeclareMathOperator{\dist}{dist}
\DeclareMathOperator{\diam}{diam}

\newcommand{\eps}{\varepsilon}

\begin{document}

\title[A convolution estimate]{A convolution estimate for
  two-dimensional hypersurfaces}

\author[I. Bejenaru]{Ioan Bejenaru} \address[Ioan Bejenaru]{Department
  of Mathematics\\University of Chicago\\Chicago, IL 60637\\USA}

\email{bejenaru@math.uchicago.edu}

\author[S. Herr]{Sebastian Herr} \address[Sebastian Herr]{Department
  of Mathematics\\University of California\\Berkeley, CA
  94720-3840\\USA} \email{herr@math.berkeley.edu}

\author[D. Tataru]{Daniel Tataru} \address[Daniel Tataru]{Department
  of Mathematics\\University of California\\Berkeley, CA
  94720-3840\\USA} \email{tataru@math.berkeley.edu}


\begin{abstract}
  Given three transversal and sufficiently regular
  hypersurfaces in $\R^3$ it follows from work of Bennett--Carbery--Wright that the convolution of two
  $L^2$ functions supported of the first and second hypersurface,
  respectively, can be restricted to an $L^2$ function on the third
  hypersurface, which can be considered as a nonlinear version of the Loomis--Whitney inequality.
  We generalize this result to a class of $C^{1,\beta}$ hypersurfaces in $\R^3$, under scaleable assumptions. The
  resulting uniform $L^2$ estimate has applications to
  nonlinear dispersive equations.
\end{abstract}

\maketitle

\section{Setup and main result}\label{sect:main}
\noindent
Given three coordinate hyperplanes $\Sigma_1$, $\Sigma_2$, $\Sigma_3$
in $\R^3$, namely
\begin{equation*}
  \Sigma_1=yz-\text{plane}, \;
  \Sigma_2=xz-\text{plane}, \; \Sigma_3=xy-\text{plane},
\end{equation*}
and smooth functions $f \in L^p(\Sigma_1)$, $g \in L^q(\Sigma_2)$,
consider estimates of the form
\[
\|f\ast g \|_{L^{r'}(\Sigma_3)}\leq C
\|f\|_{L^p(\Sigma_1)}\|g\|_{L^q(\Sigma_2)}.
\]
Since
\begin{equation*}
  (f\ast g) (x,y,z)=\int f(y,z') g(x,z-z') dz'
\end{equation*}
by duality the above estimate is equivalent to
\begin{equation*}
  \left| \int f(y,z)g(x,-z)h(x,y)dxdydz \right| \leq C \|f\|_{L^p(\Sigma_1)}\|g\|_{L^q(\Sigma_2)}\|h\|_{L^{r}(\Sigma_3)}.
\end{equation*}
By H\"older's inequality, we obtain the necessary and sufficient
conditions $p=q=r=2$.  In that case we obtain the bound
\begin{equation}\label{eq:conv_gen_lp}
  \|f\ast g \|_{L^{2}(\Sigma_3)}\leq C \|f\|_{L^2(\Sigma_1)}
  \|g\|_{L^2(\Sigma_2)}.
\end{equation}
which is the classical Loomis-Whitney inequality in three 
space dimensions, see \cite{MR0031538}.

The question which we address here is the following: Does the estimate
\eqref{eq:conv_gen_lp} remain true if $\Sigma_1$, $\Sigma_2$ and
$\Sigma_3$ are bounded subset of transversal, sufficiently smooth,
and oriented surfaces in $\R^3$?

This question has been answered in the affirmative in \cite[Proposition 7]{MR2155223}, along with a quantitative estimate, under the assumption of $C^3$ regularity and a local transversality condition on the surfaces.
In considering this question we are motivated by problems which arise
in the analysis of bilinear $X^{s,b,p}$ estimates in various nonlinear
dispersive equations. Precisely, one can view the estimate
\eqref{eq:conv_gen_lp} as a limiting case of the following bound:
\[
\| f g\|_{X^{0,-\frac12,\infty}_{\Sigma_3}} \leq C
\|f\|_{X^{0,\frac12,1}_{\Sigma_1}} \|g\|_{X^{0,\frac12,1}_{\Sigma_2}}.
\]
Here $f,g$ are assumed to have Fourier support supported in a fixed
unit ball, and the norms $X^{0,\frac12,1}_{\Sigma}$, respectively
$X^{0,-\frac12,\infty}_{\Sigma}$ are defined by
$$
\| f\|_{X^{0,\frac12,1}_{\Sigma}} = \sum_{k = -\infty}^0
2^{\frac{k}{2}}\| \hat{f}(\xi) \|_{L^2(\{ 2^{k} \leq \dist(\xi,\Sigma)
  \leq 2^{k+1}\})},
$$
respectively
$$
\| h\|_{X^{0,-\frac12,\infty}_{\Sigma}} = \sup_{k \leq 0}
2^{-\frac{k}{2}}\| \hat{f}(\xi)\|_{L^2(\{ 2^{k} \leq \dist(\xi,\Sigma)
  \leq 2^{k+1}\})}.
$$
By rescaling, this implies estimates on dyadic frequency scales, in the low modulation region.
Bounds of this type have already appeared --- at least implicitely ---
in the study of bilinear interactions in many semilinear equations
with nontrivial resonance sets, i.e. when bilinear interactions of
solutions to the linear homogeneous equation have an output near the
characteristic set, which in our context means that
$\Sigma_1+\Sigma_2$ has a nontrivial intersection with $\Sigma_3$.

For instance, in the context of Schr\"odinger equations we can mention
\cite{CDKS01}, \cite[Lemma 4.1]{BS08} and \cite{B08}; there the three
surfaces are (pieces of) parabolas. Other examples are the bounds for
the KP-I equation considered in \cite{IKT07}. A large class of
bilinear and multilinear estimates have been systematically studied in
\cite{T01}; however, this does not include the present setup.

In most applications to dispersive equations bilinear estimates are proved in an
ad-hoc manner, taking advantage of the exact form of the surfaces
$\Sigma_1$, $\Sigma_2$ and $\Sigma_3$. In all cases mentioned above the three
surfaces have nonvanishing curvature, and one may ask which (if any)
is the role played by the curvature in general. It is well known that
the nonvanishing curvature plays a fundamental role in the study of
nonlinear dispersive equations, as it insures good decay properties
for the fundamental solution of the corresponding linear equation, as
well as Strichartz and other estimates for solutions to the linear
equation. On the other hand, the bound \eqref{eq:conv_gen_lp} is still
valid when $\Sigma_1$, $\Sigma_2$ and $\Sigma_3$ are transversal planes.
Note that the role of curvature and transversality has been
clarified in \cite{MR2155223} in a much broader context:
Curvature is dispensable for the validity of estimate \eqref{eq:conv_gen_lp}.
However, it still is desireable to gain a better
understanding of the interplay of the necessary size, regularity, transversality and curvature assumptions
on the surfaces under which sharp quantitative and scaleable estimates of the type \eqref{eq:conv_gen_lp} hold true.
Our main motivation are bilinear estimates with applications to the 2d Zakharov system, see \cite{BHHT08}, in which case we need to analyze the interaction between two paraboloids and a cone.

Our precise set-up is the following.
\begin{assumption}\label{ass}
For $i=1,2,3$ there exists $0<\beta\leq 1$, $b>0$, and $\theta>0$ as well as\footnote{We will use the larger surfaces $\Sigma_i^*$ and condition \ref{it:reg_cond} only to ensure the existence of
a global representation of $\Sigma_i$ as a graph over a cube in a suitable frame.
This assumption is likely to be redundant, but we will not pursue these matters here as it is irrelevant for applications.}
$\Sigma_i^*\subset \R^3$, such that $\Sigma_i$ is an open and bounded subset
 of $\Sigma_i^*$ and
\begin{enumerate}
\item\label{it:reg_cond}
the oriented surface $\Sigma_i^*$ is given as
\begin{equation}\label{eq:reg_cond}
\Sigma_i^*=\{\sigma_i \in U_i \mid \Phi_i(\sigma_i)=0, \nabla \Phi_i \not=0, \Phi_i \in C^{1,\beta}(U_i)\},
\end{equation}
for a convex $U_i\subset \R^3$ such that $\dist(\Sigma_i,U_i^c)\geq \diam(\Sigma_i)$;

\item\label{it:hoeldercond}
the unit normal vector field $\mathfrak{n}_i$ on $\Sigma_i^*$ satisfies
the H\"older condition
\begin{equation}\label{eq:hoeldercond}
  \sup_{\sigma,\tilde{\sigma} \in \Sigma_i^*}
  \frac{|\mathfrak{n}_i(\sigma)-\mathfrak{n}_i(\tilde{\sigma})|  }{|\sigma-\tilde{\sigma}|^{\beta}}+  \frac{|\mathfrak{n}_i(\sigma)(\sigma-\tilde{\sigma})|  }{|\sigma-\tilde{\sigma}|^{1+\beta}}\leq b;
\end{equation}
\item\label{it:trans}
the matrix $N(\sigma_1,\sigma_2,\sigma_3)=(\mathfrak{n}_1(\sigma_1),\mathfrak{n}_2(\sigma_2),\mathfrak{n}_3(\sigma_3))$
satisfies the transversality condition
\begin{equation}
  \label{eq:trans}
  \theta \leq\det N(\sigma_1,\sigma_2,\sigma_3) \leq 1
\end{equation}
for all
$(\sigma_1,\sigma_2,\sigma_3)\in \Sigma_1^*\times \Sigma_2^*\times
\Sigma_3^*$.
\end{enumerate}
\end{assumption}

\begin{figure}[ht]
\centering
  \begin{picture}(0,0)%
    \includegraphics[width=11cm]{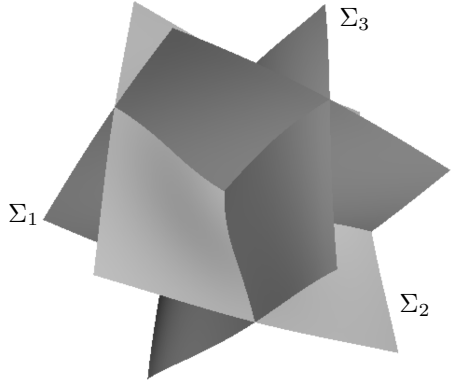}%
  \end{picture}%
  \setlength{\unitlength}{1cm}%
  \begin{picture}(11,6)(0.1,0.1)
    \put(2.6,3){\makebox(0,0)[lb]{\smash{\footnotesize$\Sigma_1$}}}
    \put(7.8,1.8){\makebox(0,0)[lb]{\smash{\footnotesize$\Sigma_2$}}}
    \put(7,5.6){\makebox(0,0)[lb]{\smash{\footnotesize$\Sigma_3$}}}
  \end{picture}
  \caption{\label{fig:surfaces} Three surfaces $\Sigma_1$, $\Sigma_2$,
    and $\Sigma_3$ in a transversal position.}
\end{figure}

We identify $f\in L^{2}(\Sigma_1)=L^{2}(\Sigma_1,\mu_1)$ --- $\mu_1$
being the surface measure on $\Sigma_1$ --- with the distribution
\begin{equation*}
  f(\psi)=\int_{\Sigma_1} f(y) \psi(y) d\mu_1(y) ,\; \psi\in \mathcal{D}(\R^3).
\end{equation*}
For $f\in L^{2}(\Sigma_1),g\in L^{2}(\Sigma_2)$ with compact support
the convolution $f\ast g \in \mathcal{D}'(\R^3)$ is given by
\begin{equation*}
  (f\ast g)(\psi)=\int_{\Sigma_1}\int_{\Sigma_2}f(x)g(y)\psi(x+y)d\mu_1(x)d\mu_2(y)
  , \; \psi \in \mathcal{D}(\R^3).
\end{equation*}

A-priori this convolution is an integrable function which is only
defined almost everywhere, therefore its restriction to $\Sigma_3$ is
not well-defined. To address this issue we begin with $f \in
C_0(\Sigma_1)$ and $g \in C_0(\Sigma_2)$. Then $f \ast g \in
C_0(\R^3)$ and has a well-defined trace on $\Sigma_3$. If
\eqref{eq:conv_gen_lp} is proved in this case, then the trace of $f
\ast g $ on $\Sigma_3$ can be defined by density for arbitrary $f \in
L^{2}(\Sigma_1)$ and $g \in L^{2}(\Sigma_2)$.

Transferring the bound \eqref{eq:conv_gen_lp} from coordinate planes
to the general setting of possibly curved surfaces turns out to
be quite nontrivial.  The reason is that the convolution has an
additive structure with respect to addition in the ambient space
$\R^3$, which is lost when restricting it to curved surfaces. Our
first result is the following.

\begin{theorem}\label{thm:conv}
  Let $\Sigma_1,\Sigma_2,\Sigma_3$ be surfaces
  in $\R^3$ which satisfy Assumption \ref{ass} with parameters $0<\beta\leq 1$, $b=1$ and
  $\theta = \frac12$, and $\diam{\Sigma_i}\leq 1$.  Then for each $f\in L^2(\Sigma_1)$ and
  $g\in L^2(\Sigma_2)$ the restriction of the convolution $f\ast g$ to
  $\Sigma_3$ is a well-defined $L^2(\Sigma_3)$-function which
  satisfies
  \begin{equation}\label{eq:conv}
    \|f\ast g\|_{L^2(\Sigma_3)}\leq C  \|f\|_{L^2(\Sigma_1)}\|g\|_{L^2(\Sigma_2)},
  \end{equation}
  where the constant $C$ depends only on $\beta$.
\end{theorem}

As mentioned above, in the case where the surfaces are of class $C^{3}$ and a local transversality condition is satisfied a variant of this theorem has been obtained earlier with a different proof. More precisely, estimate \eqref{eq:conv} can be derived from \cite[Theorem 1]{MR2155223} as in the proof of \cite[Proposition 7 or Theorem 8]{MR2155223}.

For multilinear estimates with applications to nonlinear dispersive equations it is necessary to make explicit how $C$ depends on the diameter of the surfaces and on $\theta$, $b$ and $\beta$. The subsequent Corollaries \ref{cor:bt}, \ref{cor:conv_gen} and in particular \ref{cor:conv_theta} are quantitative refinements of the above Theorem which --- to the best of our knowledge --- are not available in the literature.

The proof of Theorem \ref{thm:conv} presented here, which merely uses $C^{1,\beta}$ regularity, is based on the induction on scales argument \'a la Bourgain, Wolff, Tao seems to be more robust compared to the proof given in \cite{MR2155223} in the sense that it does not require the normals to be Lipschitz. On the other hand, the induction on scales machinery has been implemented in \cite{MR2275834} in a more general context, but the results of \cite{MR2275834} imply our results only up to a small loss in the induction on scales procedure, see also \cite[Remark 6.3]{MR2275834}. Note that the homogeneous regularity assumption \eqref{eq:hoeldercond} allows us to take advantage of the isotropic scaling, which is essential for the derivation of the subsequent Corollaries.

The result \eqref{eq:conv} can be viewed as a weaker form of the 
three dimensional multilinear restriction conjecture, see \cite{MR2275834}. 
Denoting
\[
\mathcal E_j f_j = \int_{\Sigma_j} e^{i x\xi} f_j(\xi) d \xi,
\qquad f_j \in L^1(\Sigma_j)
\]
we have 

\begin{conjecture}\cite{MR2275834} Assume that $\Sigma_1$, $\Sigma_2$ and
  $\Sigma_3$ satisfy the transversality condition \ref{it:trans} above.
Then 
\begin{equation}
\| \mathcal E_1 f_1 \mathcal E_2 f_2 \mathcal E_3 f_3 \|_{L^1}
\leq C \|f_1\|_{L^2(\Sigma_1)} \|f_2\|_{L^2(\Sigma_2)} 
\|f_3\|_{L^2(\Sigma_3)} 
\label{multrest}\end{equation}
\end{conjecture}

With an $\epsilon$ loss this is proved in \cite{MR2275834},
\begin{equation}
\| \mathcal E_1 f_1 \mathcal E_2 f_2 \mathcal E_3 f_3 \|_{L^1(B(0,R))}
\leq C R^\epsilon \|f_1\|_{L^2(\Sigma_1)} \|f_2\|_{L^2(\Sigma_2)} 
\|f_3\|_{L^2(\Sigma_3)}.
\label{multresta}\end{equation}

Another generalization of the Loomis-Whitney inequality is given
in \cite{MR2155223}. In the context of the above restriction conjecture,
the results in \cite{MR2155223} imply that
\[
\| \mathcal E_1 f_1 \mathcal E_2 f_2 \mathcal E_3 f_3 \|_{L^2}
\leq C \|f_1\|_{L^{4/3}(\Sigma_1)} \|f_2\|_{L^{4/3}(\Sigma_2)} 
\|f_3\|_{L^{4/3}(\Sigma_3)} 
\]
which would follow from \eqref{multrest} by multilinear interpolation
with the trivial $L^\infty$ bound for the product.

We conclude the section with a discussion of further versions of our
main result.  Partitioning the three surfaces into smaller pieces and
using linear changes of coordinates it is easy to allow arbitrary
values for the parameters in the hypothesis of the theorem:

 \begin{cor}\label{cor:bt}
   Let $\Sigma_1,\Sigma_2,\Sigma_3$ be surfaces
   in $\R^3$ which satisfy Assumption \ref{ass} with parameters $0<\beta\leq 1$, $b>0$, $\theta>0$,
   and $\diam{\Sigma_i}\leq R$.  Then for each $f\in
   L^2(\Sigma_1)$ and $g\in L^2(\Sigma_2)$ the restriction of the
   convolution $f\ast g$ to $\Sigma_3$ is a well-defined
   $L^2(\Sigma_3)$-function which satisfies
   \begin{equation}\label{eq:bt}
     \|f\ast g\|_{L^2(\Sigma_3)}\leq C(R^\beta b,\theta)  \|f\|_{L^2(\Sigma_1)}\|g\|_{L^2(\Sigma_2)}.
   \end{equation}
 \end{cor}

 Here the expression $R^\beta b$ appears due to isotropic scaling.
 While this is easy to prove, it is not so useful due to the
 unspecified dependence of the constant $C$ on $R^\beta b$ and
 $\theta$.  A better result is contained in the next Corollary, which
 considers the case of three surfaces which can be placed into the
 context of Theorem~\ref{thm:conv} via a linear transformation.

\begin{cor}\label{cor:conv_gen}
  Assume that $\Sigma_1$, $\Sigma_2$,
  $\Sigma_3$ satisfy the conditions of Theorem \ref{thm:conv}.  Let
  $T:\R^3\to \R^3$ be an invertible, linear map and
  $\Sigma_i'=T\Sigma_i$. Then for functions $f'\in L^2(\Sigma_1')$ and
  $g'\in L^2(\Sigma_2')$ the restriction of the convolution $f'\ast
  g'$ to $\Sigma_3'$ is a well-defined $L^2(\Sigma_3')$-function which
  satisfies
  \begin{equation}\label{eq:conv_gen}
    \|f'\ast g'\|_{L^2(\Sigma_3')}\leq \frac{C}{\sqrt{d}} \|f'\|_{L^2(\Sigma_1')}\|g'\|_{L^2(\Sigma_2')},
  \end{equation}
  where
  \begin{equation*}
    d=\inf_{\sigma_1',\sigma_2',\sigma_3'}|\det N'(\sigma_1',\sigma_2',\sigma_3')|
  \end{equation*}
  and $N'(\sigma_1',\sigma_2',\sigma_3')$ is the matrix of the unit
  normals to $\Sigma_i'$ at $(\sigma_1',\sigma_2',\sigma_3')$ and $C$ depends only on $\beta$.
\end{cor}

We remark that the linear transformation $T$ does not explicitely
appear in the estimate \eqref{eq:conv_gen}. Instead, the size
$\frac{1}{\sqrt{d}}$ of the constant is determined only by the
transversality properties of the surfaces
$\Sigma_1',\Sigma_2',\Sigma_3'$. Hence the best way to interpret the
result in the Corollary is to say that the bound \eqref{eq:conv_gen}
for the surfaces $\Sigma_1',\Sigma_2',\Sigma_3'$ holds whenever these
surfaces are bounded, $C^{1,\beta}$ regular and uniformly transversal
with respect to some linear frame.

Finally, let us state an explicit condition which guarantees that the
assumptions of Corollary \ref{cor:conv_gen} are satisfied:

\begin{cor}\label{cor:conv_theta}
  Let $\Sigma_1,\Sigma_2,\Sigma_3$ be
  surfaces in $\R^3$ which satisfy Assumption \ref{ass} with parameters $0<\beta \leq 1$, $b>0$, $\theta>0$,
  and $\diam{\Sigma_i}\leq R$, so that
  \begin{equation}
    R^\beta b \leq    \theta.
    \label{rbt}\end{equation}
  Then for each $f\in L^2(\Sigma_1)$ and $g\in L^2(\Sigma_2)$ the
  restriction of the convolution $f\ast g$ to $\Sigma_3$ is a
  well-defined $L^2(\Sigma_3)$-function which satisfies
  \begin{equation}\label{eq:conv_theta}
    \|f\ast g\|_{L^2(\Sigma_3)}\leq  \frac{C}{\sqrt{\theta}}
    \|f\|_{L^2(\Sigma_1)}\|g\|_{L^2(\Sigma_2)},
  \end{equation}
where $C$ depends only on $\beta$.
\end{cor}

Finally, we remark that the factor $\theta^{-\frac12}$ which appears in \eqref{eq:conv_theta} has also been obtained in \cite[Theorem 1.2]{MR2155223}, but with an unspecified dependence of $R$ on $b$ and $\theta$.

\section{Linear changes of coordinates}\label{sect:reduction_main}
\noindent

The aim of this section is to use linear transformations in order to
derive Corollaries~\ref{cor:conv_gen},\ref{cor:conv_theta} from
Theorem \ref{thm:conv}. Parts of these arguments will also be useful
in the proof of Theorem \ref{thm:conv}.

\begin{proof}[Proof of Corollary \ref{cor:conv_gen}]
  We may assume that we have a single
  coordinate patch for each surface, i.e. there is a global
  parametrization $\varphi_i: \R^2 \supset \Omega_i\to \R^3$,
  $\varphi_i(\Omega_i)=\Sigma_i$.
  Therefore, $\Sigma_i'=T\Sigma_i$ is
  parametrized by $\varphi_i'=T\varphi_i$.  For $i=1,2,3$ we define
  the induced normals
  \begin{equation*}
    m_i(u)=\partial_1 \varphi_i(u)\times \partial_2 \varphi_i(u), \qquad m'_i(u)=\partial_1 \varphi'_i(u)\times \partial_2 \varphi'_i(u).
  \end{equation*}
  It is easily checked that
  \begin{equation}\label{eq:change_induced_normals}
    m_i'=\det(T)(T^{-1})^{\top} m_i
  \end{equation}
  and therefore
  \begin{equation}\label{eq:change_length1}
    \begin{split}
      |m_1'||m_2'||m_3'|=&\ |\det(m_1',m_2',m_3')||\det(\mathfrak{n}_1',\mathfrak{n}_2',\mathfrak{n}_3')|^{-1}\\
      =&\ \frac{\det(T)^2\det(\mathfrak{n}_1,\mathfrak{n}_2,\mathfrak{n}_3)}{|\det(\mathfrak{n}_1',\mathfrak{n}_2',\mathfrak{n}_3')|}|m_1||m_2||m_3|
    \end{split}
  \end{equation}
  Let $f'\in L^2(\Sigma_1')$, $g'\in L^2(\Sigma_2')$ and $h'\in
  L^2(\Sigma_3')$ be given and it is enough to consider non-negative functions. We write $f=f'(T\cdot)$, and
  \begin{equation*}
    \tilde{f'}(u)=f'(\varphi'_1(u))|m'_1(u)|^{\frac12},
  \end{equation*}
  and similarly for $g'$, $h'$. With this notation and by duality, our
  claim is equivalent to
  \begin{equation}\label{eq:rewrite_conv_gen1}
    \begin{split}
      &\left|\int
        \tilde{f'}(u)\tilde{g'}(v)\tilde{h'}(w)|m_1'(u)|^{\frac12}|m_2'(v)|^{\frac12}|m_3'(w)|^{\frac12}
        d\nu'(u,v,w)\right|\\
      & \qquad \qquad \leq
      \frac{C}{\sqrt{d}}\|\tilde{f'}\|_{L^2(\Omega_1)}\|\tilde{g'}\|_{L^2(\Omega_2)}\|\tilde{h'}\|_{L^2(\Omega_3)}
    \end{split}
  \end{equation}
  with the measure
  \begin{equation*}
    d\nu'(u,v,w)=\delta(\varphi'_1(u)+\varphi'_2(v)-\varphi'_3(w))dudvdw,
  \end{equation*}
  where $\delta$ denotes the Dirac delta distribution at the origin in
  $ \R^3$.  Let us also define the measure
  \begin{equation*}
    d\nu(u,v,w)=\delta(\varphi_1(u)+\varphi_2(v)-\varphi_3(w))dudvdw.
  \end{equation*}
  Using $\delta(T\cdot)=|\det(T)|^{-1}\delta(\cdot)$, and
  \eqref{eq:change_length1} we have that
  \begin{equation*}
    |m_1'|^{\frac12}|m_2'|^{\frac12}|m_3'|^{\frac12}d\nu'
    \sim\frac{|m_1|^{\frac12}|m_2|^{\frac12}|m_3|^{\frac12}}{|\det(\mathfrak{n}_1',\mathfrak{n}_2',\mathfrak{n}_3')|^{\frac12}}d\nu
  \end{equation*}
  by the transversality assumption \eqref{eq:trans} with $\theta=\frac12$ on the normals to
  $\Sigma_i$.  Therefore \eqref{eq:rewrite_conv_gen1} is equivalent to
  \begin{equation}\label{eq:rewrite_conv_gen2}
    \begin{split}
      &\left|\int \tilde{f'}(u)\tilde{g'}(v)\tilde{h'}(w)\frac{|m_1(u)|^{\frac12}|m_2(v)|^{\frac12}|m_3(w)|^{\frac12}d\nu(u,v,w)}{|\det(\mathfrak{n}_1'(\varphi'_1(u)),\mathfrak{n}_2'(\varphi_2'(v)),\mathfrak{n}_3'(\varphi_3'(w)))|^{\frac12}}\right|\\
      & \qquad \qquad \leq
      \frac{C}{\sqrt{d}}\|\tilde{f'}\|_{L^2(\Omega_1)}\|\tilde{g'}\|_{L^2(\Omega_2)}\|\tilde{h'}\|_{L^2(\Omega_3)}.
    \end{split}
  \end{equation}
  Observe that it follows from our assumption that the corresponding
  estimate for $\Sigma_1,\Sigma_2,\Sigma_3$, namely
  \begin{align*}
    &\left|\int \tilde{f'}(u)\tilde{g'}(v)\tilde{h'}(w)|m_1(u)|^{\frac12}|m_2(v)|^{\frac12}|m_3(w)|^{\frac12}d\nu(u,v,w)\right|\\
    & \qquad \qquad \leq
    C\|\tilde{f'}\|_{L^2(\Omega_1)}\|\tilde{g'}\|_{L^2(\Omega_2)}\|\tilde{h'}\|_{L^2(\Omega_3)}
  \end{align*}
  holds true. This implies \eqref{eq:rewrite_conv_gen2}.
\end{proof}

\begin{proof}[Proof of Corollary \ref{cor:conv_theta}]
  Partitioning each of the three surfaces into smaller sets we
  strenghten the relation \eqref{rbt} to
  \begin{equation}
    R^\beta b \leq  2^{-10}  \theta.
    \label{rbta}\end{equation}
  Consider a fixed triplet $(\sigma_1^0,\sigma_2^0,\sigma_3^0)\in
  \Sigma_1\times \Sigma_2\times\Sigma_3$. For arbitrary
  $(\sigma_1,\sigma_2,\sigma_3)\in \Sigma_1\times
  \Sigma_2\times\Sigma_3$ we use the H\"older condition to compute
  \begin{equation}
    |\mathfrak{n}_i(\sigma_i) - \mathfrak{n}_i (\sigma_i^0)| \leq b
    R^{\beta} \leq 2^{-10} \theta, \qquad i = 1,2,3
    \label{ndif}\end{equation}
  This implies that $N(\sigma_1,\sigma_2,\sigma_3)$ does not vary
  much,
  \begin{equation}
    |\det N(\sigma_1^0,\sigma_2^0,\sigma_3^0) -
    \det N(\sigma_1,\sigma_2,\sigma_3)| \leq 2^{-8} \theta
    \label{detndif}\end{equation}
  Hence, after possibly increasing $\theta$, we may assume that on $
  \Sigma_1\times \Sigma_2\times\Sigma_3$ we have the stronger bound
  \[
  \theta \leq |N(\sigma_1,\sigma_2,\sigma_3)| \leq (1+2^{-8}) \theta.
  \]
  From the H\"older condition we also obtain
  \[
  |(\sigma_i -\sigma_i^0) \mathfrak{n}_i^0| \leq 2^{-10} R \theta
  \]
  which shows that $\Sigma_i$ is contained in an infinite slab of
  thickness $ 2^{-10} R \theta$ with respect to the $\mathfrak{n}_i^0$
  direction. By orthogonality with respect to such slabs it suffices
  to prove the desired bound \eqref{eq:conv_theta} in the case when
  the other two surfaces are contained in similar slabs,
  \begin{equation}
    |(\sigma_i -\sigma_i^0)  \mathfrak{n}_j^0| \leq 2^{-10} R \theta,
    \qquad i,j = 1,2,3.
    \label{planeloc}\end{equation}
  We will apply Corollary \ref{cor:conv_gen} with
  \begin{equation*}
    T=  R \theta 
    (N^\top)^{-1}, \qquad N = N(\sigma_1^0,\sigma_2^0,\sigma_3^0).
  \end{equation*}
  We need to show that for $\tilde{\Sigma}_i:=T^{-1}\Sigma_i$ we have
  the conditions \eqref{eq:trans} with $b=1$, \eqref{eq:hoeldercond}
  with $\theta = \frac12$ and the size condition
  $\diam(\tilde{\Sigma}_i)\leq 1$.  Concerning the latter we observe
  that
  \[
  T^{-1} (\sigma_i -\sigma_i^0) = \frac{1}{R \theta} \left(
    \begin{array}{c} \mathfrak{n}_1^0 (\sigma_i -\sigma_i^0) \cr
      \mathfrak{n}_2^0 (\sigma_i -\sigma_i^0) \cr \mathfrak{n}_3^0
      (\sigma_i -\sigma_i^0) \end{array} \right)
  \]
  Thus by \eqref{planeloc} we obtain
  \[
  |T^{-1} (\sigma_i -\sigma_i^0)| \leq \sqrt{3} 2^{-10} \leq \frac12
  \]
  therefore $\diam(\tilde{\Sigma}_i)\leq 1$.  For the transversality
  condition, we first estimate
  \begin{equation}\label{eq:norm_inversea}
    \|N^{-1}\|\leq 2
    |\det{N}|^{-1}\|N\|^2 \leq 6 \theta^{-1} 
  \end{equation}  
  This gives a bound for $T$, namely
  \begin{equation}\label{eq:norm_inverseb}
    \|T\|\leq  6 R.
  \end{equation}
  The unit normal $\tilde{\mathfrak{n}}_i(\tilde{\sigma}_i)$ to
  $\tilde{\Sigma}_i$ in $\tilde{\sigma}_i\in \tilde{\Sigma}_i$ is
  given by
  \begin{equation}\label{eq:change_normals}
    \tilde{\mathfrak{n}}_i(\tilde{\sigma}_i)=
    \frac{N^{-1}\mathfrak{n}_i(T\tilde{\sigma}_i)}{|N^{-1}\mathfrak{n}_i(T\tilde{\sigma}_i)|}.
  \end{equation}
  By construction for $\tilde\sigma_i^0 = T^{-1} \sigma_i^0$ we have $
  \tilde{\mathfrak{n}}_i(\tilde{\sigma}_i^0)=N^{-1}{\mathfrak{n}}_i({\sigma}_i^0)
  = e_i$. By \eqref{ndif} and \eqref{eq:norm_inversea} it follows that
  \[
  |N^{-1}\mathfrak{n}_i(T\tilde{\sigma}_i) -e_i| \leq 2^{-7}
  \]
  which implies that
  \[
  |\tilde{\mathfrak{n}}_i(T\tilde{\sigma}_i) -e_i| \leq 2^{-5}.
  \]
  This in turn yields the desired transversality condition
  \[
  \det
  \tilde{N}(\tilde{\sigma}_1,\tilde{\sigma}_2,\tilde{\sigma}_3)\geq
  1/2.
  \]
  Finally, for the H\"older condition we use \eqref{eq:norm_inversea}
  and \eqref{eq:norm_inverseb} to compute
  \begin{equation*}
    \frac{|\tilde{\mathfrak{n}}_i(\tilde{\sigma})-\tilde{\mathfrak{n}}_i(\tilde{\rho})|}{|\tilde{\sigma}-\tilde{\rho}|^{\beta}}
    \leq 3 \|N^{-1}\|\|T\|^{\beta}\frac{|\mathfrak{n}_i(T\tilde{\sigma})
      -\mathfrak{n}_i(T\tilde{\rho})|}{|T\tilde{\sigma}-T\tilde{\rho}|^{\beta}}\leq 
    2^7  \theta^{-1} R^\beta b \leq 1
  \end{equation*}
  which proves the desired bound for the first term in the H\"older
  condition \eqref{eq:hoeldercond} with $b=1$. The second term in
  \eqref{eq:hoeldercond} is treated in the same way.
\end{proof}

\section{Induction on scales}
\noindent
Theorem \ref{thm:conv} is obtained from uniform estimates for $f\ast
g$ thickened surfaces $\Sigma_3(\eps)$ given by
\begin{equation} \label{epsdef} \Sigma_3(\eps):=\{ v \in \R^3 \mid
  \dist(v,\Sigma_3)\leq \eps\}.
\end{equation}
Our main technical result is the following.
\begin{prop}\label{prop:eps_thick}
  For all $\Sigma_1$, $\Sigma_2$, and $\Sigma_3$ with $\diam(\Sigma_i)\leq 1$ as in Theorem \ref{thm:conv}
  and $f\in L^2(\Sigma_1)$, $g\in L^2(\Sigma_2)$ and $0 < \eps \leq 1$
  the estimate
  \begin{equation}\label{eq:eps_thick}
    \left\|f \ast g \right\|_{L^2(\Sigma_3(\eps))}
    \leq C \sqrt{\eps} \|f\|_{L^2(\Sigma_1)}\|g\|_{L^2(\Sigma_2)}
  \end{equation}
  holds true with a constant $C$ depending only on $\beta$.
\end{prop}

We first show how this implies the main Theorem.

\begin{proof}[Proof of Theorem \ref{thm:conv}]
  For $f \in C_0(\Sigma_1)$ and $g \in C_0(\Sigma_2)$ we have $f \ast
  g \in C_0(\R^3)$. Then
  \[
  \| f \ast g \|_{L^2(\Sigma_3)} = \lim_{\eps \rightarrow 0}
  \eps^{-\frac12} \| f \ast g \|_{L^2(\Sigma_3(\eps))},
  \]
  therefore \eqref{eq:conv} follows from \eqref{eq:eps_thick}.  The
  result extends by density to $f \in L^2(\Sigma_1)$ and $g \in
  L^2(\Sigma_2)$.
\end{proof}

The rest of this section is devoted to the proof of Proposition
\ref{prop:eps_thick}.
By repeating the argument from the proof of Corollary
\ref{cor:conv_theta} --- namely a finite partition of the surfaces
$\Sigma_i$, scaling and transforming the normals at one triplet of
points to $\mathfrak{e}_1, \mathfrak{e}_2, \mathfrak{e}_3$ --- we can
reduce Proposition \ref{prop:eps_thick} to the following setup:

There are unit cubes $\Omega_i\subset \R^2$, $i=1,2,3$ centered at
points $a_i^0 \in \R^2$, and $C^{1,\beta}$ functions $\phi_i$ in the
doubled cubes $2\Omega_i$ so that
\begin{equation}\label{eq:surface_res}
  \begin{split}
    \Sigma_1=&\{(x,y,z)\in \R^3 \mid (y,z) \in \Omega_1, \; x=\phi_1(y,z)\},\\
    \Sigma_2=&\{(x,y,z)\in \R^3\mid (x,z) \in \Omega_2, \; y=\phi_2(x,z)\},\\
    \Sigma_3=&\{(x,y,z)\in \R^3 \mid (x,y) \in \Omega_3,\;
    z=\phi_3(x,y)\},
  \end{split}
\end{equation}
where the functions $\phi_i$ satisfy
\begin{equation}\label{eq:fix}
  \nabla \phi_1(a_1^0)=\nabla\phi_2(a_2^0)=\nabla\phi_3(a_3^0)=0 
\end{equation}
and have small H\"older constant
\begin{equation} \label{s1bnew} \sup_{w,\tilde{w}\in 2\Omega_i}\frac{|
    \nabla \phi_i (w) - \nabla \phi_i {(\tilde{w})} |}{|w -
    \tilde{w}|^\beta} \leq 2^{-40}.
\end{equation}

To set up the induction on scales we allow the scale of the cubes
$\Omega_i$ to vary from $0$ to $1$. Precisely, for $ 0 < \eps \leq r
\leq 1$ we denote by $C(r,\eps)$ the best constant $C$ in the estimate
\begin{equation}\label{eq:eps}
  \left\|f \ast g \right\|_{L^2(\Sigma_3(\eps))}
  \leq C \sqrt{\eps} \|f\|_{L^2(\Sigma_1)}\|g\|_{L^2(\Sigma_2)}
\end{equation}
considered over all surfaces $\Sigma_1$, $\Sigma_2$, $\Sigma_3$ as in
\eqref{eq:surface_res} with $\Omega_i$ cubes of size $r$ and $\phi_i$
satisfying \eqref{eq:fix} and \eqref{s1bnew}.

For $0 \leq \eps \leq r \leq R \leq 1$ we also introduce the auxiliary
notation $C(R,r,\eps)$ as the best constant in the estimate
\eqref{eq:eps} over all surfaces $\Sigma_1$, $\Sigma_2$, $\Sigma_3$ as
in \eqref{eq:surface_res} with $\Omega_i$ cubes of size $r$ and
$\phi_i$ satisfying \eqref{s1bnew} in larger cubes $4\Omega_i$ and a weaker
version of \eqref{eq:fix}, namely
\begin{equation}\label{eq:fixa}
  |\nabla \phi_i(a_i^0)| \leq 2^{-40} (\sqrt{3} R)^\beta 
\end{equation}

Throughout this paper (and hence in the above definitions) we agree to
the following convention: the size of a cube is half of its
side-length. The reason for doing so is purely technical as it spares
us from carrying a factor of $2$ in some estimates.

As a starting point of our induction we establish the desired bound
when $r$ is sufficiently small, depending on $\eps$.

\begin{lemma}\label{lem:start}
  Assume that $ r \leq \eps^\frac{1}{\beta+1}$.  Then $C(r,\eps)\leq
  1$.
\end{lemma}
\begin{proof}
  For $f\in L^2(\Sigma_1)$, $g \in L^2(\Sigma_2)$ and an arbitrary
  test function $\psi$ the convolution $(f\ast g) (\psi)$ can be
  expressed in the form
  \begin{equation*}
    \int \tilde{f}(y,z)\tilde{g}(x,\zeta-z)
    \psi(x+\phi_1(y,z),y+\phi_2(x,\zeta-z),\zeta) dxdydz d \zeta
  \end{equation*}
  where
  \begin{align*}
    \tilde{f}(y,z)=&f(\phi_1(y,z),y,z)\sqrt{1+|\nabla\phi_1(y,z)|^2}\\
    \tilde{g}(x,\zeta)=&g(x,\phi_2(x,\zeta),\zeta)\sqrt{1+|\nabla\phi_2(x,\zeta)|^2}.
  \end{align*}
  Assume that $\supp \psi\subset \Sigma_{3}(\eps)$. For every $v \in
  \supp \psi$ there is $a \in \Omega_3$ such that $d ((a,\phi_3(a)),
  v) \leq \eps$ and $d(a,a_3^0) \leq \sqrt{3}r$.  Using \eqref{s1bnew}
  we have:
  \[
  | \phi_3 (a) - \phi_3 (a_3^0) | \leq \sqrt{3} r \sup_{b \in
    \Omega_3} |\nabla \phi_3 (b)| \leq 2^{-10} (\sqrt{3} r)^{1+\beta}
  \leq 2^{-8} \eps
  \]
  Then,
  \begin{equation*}
    \supp \psi \subset \{(x,y,z)\mid |z-\phi_3(a_3^0)|\leq  \frac{\eps}4\}.
  \end{equation*}
  Let us denote
  $J(\eps)=[\phi_3(a_3^0)-\frac{\eps}4,\phi_3(a_3^0)+\frac{\eps}4]$.
  The Cauchy-Schwarz inequality shows that
  \begin{align*}
    | (f \ast g)(\psi)| &\leq \int
    \|\tilde{f}(\cdot,z)\|_{L^2}\|\tilde{g}(\cdot,\zeta-z)\|_{L^2}
    I(z,\zeta)dzd\zeta\\
    &\leq \|\tilde{f}\|_{L^2}\|\tilde{g}\|_{L^2}\int_{J(\eps)}
    \sup_{z} I(z,\zeta) d\zeta,
  \end{align*}
  where
  \begin{equation*}
    I(z,\zeta):=\left(\int|\psi(x+\phi_1(y,z),y+\phi_2(x,\zeta-z),\zeta)|^2dxdy.
    \right)^{\frac12}
  \end{equation*}
  By the change of variables
  \begin{equation*}
    \Phi_{z,\zeta}(x,y)=(x+\phi_1(y,z),y+\phi_2(x,\zeta-z))
  \end{equation*}
  we obtain
  \begin{equation*}
    I(z,\zeta)\leq \sqrt{2} \|\psi(\cdot,\cdot,\zeta)\|_{L^2(\R^2)},
  \end{equation*}
  because $|\det D\Phi_{z,\zeta}(x,y)|=|1-\partial_y\phi_1(y,z)
  \partial_x\phi_2(x,\zeta-z)| \geq \frac12$.  Finally, since $z$ is
  in an interval of size $\frac{\eps}2$, using again the
  Cauchy-Schwarz inequality we obtain
  \begin{equation*}
    \int_{J(\eps)} \sup_{z} I(z,\zeta) d\zeta \leq \frac1{\sqrt{2}} \sqrt{\eps}\|\psi\|_{L^2}
  \end{equation*}
  Hence it follows that
  \begin{align*}
    (f\ast g) (\psi)\leq &
    \sqrt{\eps}\|\tilde{f}\|_{L^2}\|\tilde{g}\|_{L^2}
    \|\psi\|_{L^2}\\
    \leq & \sqrt{\eps}
    \|f\|_{L^2(\Sigma_1)}\|g\|_{L^2(\Sigma_2)}\|\psi\|_{L^2(\Sigma_{3}(\eps))}.
  \end{align*}
  The bound $C(r,\eps) \leq 1$ follows by density and duality.
\end{proof}

The previous lemma shows that $C(r,\eps)$ is finite. By the argument
in the proof of Corollary \ref{cor:conv_theta} it also follows that
$C(R,r,\eps)$ is always finite.  The next lemma allows us to bound the
auxiliary variable $C(R,r,\eps)$ in terms of $C(r,\eps)$.

\begin{lemma} The following estimate holds true
  \begin{equation}\label{general}
    C (R,r,\eps) \leq (1+R^{\beta}) C(r (1+R^\beta)^{1+\frac1\beta} ,\eps (1+R^\beta)^{1+\frac1\beta})
  \end{equation}
\end{lemma}

\begin{proof} 
  The argument proceeds along the lines of the proof of Corollary
  \ref{cor:conv_theta}, carefully keeping track of the scales.  Let
  $\sigma_i^0 \in \Sigma_i$ be the images of $a_i^0 \in \Omega_i$. We
  consider the linear transformation defined by the matrix
  \begin{equation*}
    T=\lambda(N^\top)^{-1}, \qquad 
    N = N(\sigma_1^0,\sigma_2^0,\sigma_3^0), \qquad 
    \lambda=(1+ R^\beta)^{-\frac1\beta}
  \end{equation*}
  We denote $ \tilde{\Sigma}_i=T^{-1} \Sigma_i$. We will show that the
  surfaces $\tilde{\Sigma}_i$ satisfy the conditions in the definition
  of $ C(r (1+R^\beta)^{1+\frac1\beta} ,\eps
  (1+R^\beta)^{1+\frac1\beta})$.

  We denote $\tilde{\sigma}_i^0 = T^{-1} \sigma_i^0$ and let
  $\tilde{a}_j^0$ be the projections of $ \tilde{\sigma}_i^0$ on the
  coordinate planes, and $\tilde{\Omega}_i$ the corresponding
  projections of $\tilde{\Sigma}_i$. Setting $i=3$ for convenience, we
  also consider the full correspondence $a_3 \to \tilde a_3$ given by
  \[
  \begin{split}
    a_3 = (x,y) \in \Omega_3 \to &\ \sigma_3 = (x,y,\phi_3(x,y)) \in
    \Sigma_3
    \\
    \to &\ \tilde \sigma_3 = T^{-1} \sigma_3 = (\tilde x,\tilde
    y,\tilde \phi(\tilde x,\tilde y)) \in \tilde \Sigma_3
    \\
    \to &\ \tilde a_3 = (\tilde x,\tilde y) \in \tilde{\Omega}_3
  \end{split}
  \]
  and similarly for $i = 1,2$.

  By construction the matrix of the unit normals to $\tilde{\Sigma}_i$
  at $\tilde{\sigma}_i^0$ is
  \[
  \tilde{N}(\tilde{\sigma}_1^0,
  \tilde{\sigma}_2^0,\tilde{\sigma}_3^0)=I.
  \]
  This implies that the condition \eqref{eq:fix} is satisfied for the
  surfaces $ \tilde{\Sigma}_i$ at the points $\tilde a_i^0$.

  The condition \eqref{eq:fixa} shows that
  \[
  |\mathfrak n_k(\sigma_k^0) - \mathfrak e_k| \leq 2^{-10}
  (\sqrt{3}R)^\beta
  \]
  which leads to
  \begin{equation}
    \| N - I\| \leq 2^{-8} R^\beta, \qquad \| N^{-1} - I\| \leq 2^{-8} R^\beta.
    \label{nbd}\end{equation}
  Also by \eqref{eq:fixa} and \eqref{s1bnew} it follows that
  \begin{equation}
    |\nabla \phi_3(a_3)| \leq 2^{-6} R^\beta, \qquad a \in 4\Omega_3
    \label{dfibd}\end{equation}

  We claim that $\tilde{\Omega}_i$ is contained in a cube of size $r
  (1+R^\beta)^{\frac1\beta}$ centered at $\tilde{a}_j^0$.  Set $i=3$
  for convenience, and consider the canonical map $T_3$ from
  $\Omega_3$ to $\tilde \Omega_3$, defined by $T_3 a_3 = \tilde a_3$.
  Then, by \eqref{nbd} and \eqref{dfibd} the chain rule shows that
  \begin{equation}
    \left \| D T_3 -
      \lambda^{-1} I_2\right\| \leq 2^{-4} R^\beta
    \label{dt3}\end{equation}
  This implies that $\tilde \Omega_i$ is contained inside a cube
  centered at $\tilde a_i^0$ which has size
  \[
  \tilde r = r (\lambda^{-1} + \sqrt{3}\ 2^{-4} R^\beta) \leq r(1+
  R^\beta)^{1+\frac1\beta}
  \]

  Or next goal is to establish the bound \eqref{s1bnew} for $\tilde
  \phi_3$.  Define the function $\Phi_3(x,y,z)=z-\phi_3(x,y)$.  Then
  $\tilde{\sigma}_3 \in \tilde{\Sigma}_3$ iff
  $\tilde{\Phi}_3(\tilde{\sigma}_3):=\Phi_3(T\tilde{\sigma}_3)=0$. The
  implicit function theorem guarantees the existence of
  $\tilde{\phi}_3$ which satisfies
  $\tilde{\Phi}_3(\tilde{x},\tilde{y},\tilde{\phi}_3(\tilde{x},\tilde{y}))=0$.
  In addition, setting $(N^\top)^{-1}= (t_1,t_2,t_3)$, we have
  \begin{equation*}
    \nabla \tilde{\phi}_3(\tilde{y},\tilde{z})=-\frac{1}{t_3\cdot n({x},{y})}\begin{pmatrix}t_1\cdot n({x},{y})\\t_2\cdot n({x},{y})\end{pmatrix}\, ,
    \; n(x,y)=\begin{pmatrix}-\partial_1 \phi_3(x,y)\\ -\partial_2
      \phi_3(x,y) \\ 1\end{pmatrix},
  \end{equation*}
  By \eqref{nbd} and \eqref{dfibd} we obtain
  \begin{equation*}
    |(N^\top)^{-1} n({x},{y}) -\mathfrak e_3 | \leq 2^{-5} R^\beta
  \end{equation*}
  which after some elementary computations leads to
  \[
  | \nabla \tilde{\phi}_3(\tilde a_3^1) - \nabla \tilde{\phi}_3(\tilde
  a_3^2)| \leq (1+ 2^{-4} R^\beta) | \nabla {\phi}_3( a_3^1) - \nabla
  {\phi}_3( a_3^2)|
  \]
  On the other hand \eqref{dt3} shows that
  \[
  | \tilde a_3^1 - \tilde a_3^2| \geq \lambda^{-1}( 1- 2^{-4} R^\beta)
  | a_3^1- a_3^2|
  \]
  Given the value of $\lambda$ it follows that
  \[
  \begin{split}
    \frac{| \nabla \tilde{\phi}_3(\tilde a_3^1) - \nabla
      \tilde{\phi}_3(\tilde a_3^2)|} { | \tilde a_3^1 - \tilde
      a_3^2|^\beta} & \leq \frac{(1+ 2^{-4} R^\beta)}{ (1+R^\beta)( 1-
      2^{-4} R^\beta)^\beta} \frac{ | \nabla {\phi}_3( a_3^1) - \nabla
      {\phi}_3( a_3^2)|}{| a_3^1- a_3^2|^\beta } \\ &\leq 2^{-10},
  \end{split}
  \]
  hence \eqref{s1bnew} is established for the surfaces
  $\tilde{\Sigma}_i$.

  Formula \eqref{eq:change_induced_normals}, combined with
  \eqref{nbd}, shows how the surface measures on $\Sigma_1$ and
  $\Sigma_2$ change:
  \begin{equation}\label{jacobbound}
    \lambda^2(1-2^{-4}R^\beta) d\tilde{\mu}_i(\tilde{\sigma}_i')\leq d\mu_i(\sigma_i')\leq \lambda^2(1+2^{-4}R^\beta) d\tilde{\mu}_i(\tilde{\sigma}_i'), \; i=1,2.
  \end{equation}

  There is a small variation in the thickness of the third surface.  A
  direct computation based on \eqref{nbd} gives
  \begin{equation} \label{thickbound} T^{-1} \Sigma_{3}(\eps) \subset
    \tilde{\Sigma}_3(\lambda^{-1}(1+2^{-8} R^{\beta})\eps).
  \end{equation}
  Moreover, if $\tilde{\psi}(\cdot)=\psi(T \cdot)$, then
  \begin{equation}
    \frac{(1-2^{-6}R^\beta)^{\frac12}}{\lambda^{\frac32}}\|\psi\|_{L^2(\R^3)}\leq \|\tilde{\psi}\|_{L^2(\R^3)}\leq \frac{(1+2^{-6}R^\beta)^{\frac12}}{\lambda^{\frac32}}\|\psi\|_{L^2(\R^3)}.
  \end{equation}
  From all the above considerations it follows that
  \begin{align*}
    &C (R,r,\eps)\\ \leq{} & (1+2^{-4} R^{\beta})^\frac52
    (1-2^{-6}R^\beta)^{-\frac12}C(r(1+R^\beta)^{1+\frac1{\beta}},\eps(1+2^{-8}
    R^{\beta})\lambda^{-1}),
  \end{align*}
  and the bound \eqref{general} follows immediately since $R \leq 1$.
\end{proof}

The following result establishes the key estimate needed for the
induction on scales argument.
\begin{prop} \label{mainprop} Assume that
  $(2^{40}\eps)^{\frac2{\beta+2}} \leq R $ and $R_\beta:=
  R^\frac{\beta+2}{2(\beta+1)}$. Then, the estimate
  \begin{equation} \label{indest1} C(R_\beta,\eps) \leq (1+
    R^\frac\beta4 ) C(R_\beta, R , \eps)
  \end{equation}
  holds true.
\end{prop}
\begin{proof}We split the proof in five steps.

  {\it Step 1. (Symmetrization)} The numbers $C(R_\beta,\eps)$ and
  $C(R, R_\beta, \eps)$ are defined by \eqref{eq:eps} (with the
  additional constraints on $\Sigma_i$).  That formula has the
  disadvantage of not revealing the symmetry of the problem with
  respect to the role of the three surfaces. However, \eqref{eq:eps}
  is equivalent to
  \[
  \langle f_1 \ast f_2,f_3\rangle \leq C \sqrt{\eps}\|f_1
  \|_{L^2(\Sigma_{1})}\|f_2\|_{L^2(\Sigma_{2})} \| f_3
  \|_{L^2(\Sigma_{3}(\eps))}
  \]
  for all $f_3\in L^2(\R^3)$. Upon replacing $f_3$ by $f_3(-\cdot)$
  and $\Sigma_3$ by $-\Sigma_3$ and considering smooth, compactly
  supported $f_3$, this coincides with the triple convolution of the
  distributions $f_1$, $f_2$ with $f_3$ at zero, i.e.
  \begin{equation} \label{sym} ( f_1 \ast f_2 \ast f_3 ) (0) \leq
    C(R,\eps) \sqrt{\eps}\|f_1
    \|_{L^2(\Sigma_{1})}\|f_2\|_{L^2(\Sigma_{2})} \| f_3
    \|_{L^2(\Sigma_{3}(\eps))}.
  \end{equation}
  By density this is an equivalent definition of $C(R,\eps)$.  Since
  \eqref{sym} is symmetric\footnote{up to the thickening of
    $\Sigma_3$, which is irrelevant at this scale.} in $\Sigma_i$ we
  prefer to use this as a definition of $C(R,\eps)$.  In a similar way
  we symmetrize the definition of $C(R_\beta, R, \eps)$.

  {\it Step 2. (Reduction of scales)} From now on we assume that
  $\Sigma_i$ are defined as in \eqref{eq:surface_res} where $\Omega_i$
  are cubes of size $R_\beta$.  By translation in the coordinate
  directions we may assume that $\phi_i(a_i^0)=0$, $i=1,2,3$. From
  \eqref{s1bnew} it follows for $r_\beta:= 2^{-39}
  R^{\frac{\beta+2}{2}}$ that
  \begin{equation} \label{inclu} \Sigma_i \subset \{(x_1,x_2,x_3) \mid
    |x_i| \leq \frac{r_\beta}2 \} \subset \{ (x_1,x_2,x_3) \mid |x_i|
    \leq \frac{R}{8} \}.
  \end{equation}
  From $(2^{40}\eps)^{\frac2{\beta+2}} \leq R $ it follows that $\eps
  \leq \frac{r_\beta}2$, therefore
  \begin{equation} \label{inclu5} \Sigma_{3}(\eps) \subset
    \{(x_1,x_2,x_3) \mid |x_3| \leq r_\beta \} \subset \{(x_1,x_2,x_3)
    \mid |x_3| \leq \frac{R}{4} \}.
  \end{equation}
  Hence we are dealing with three scales ordered as follow:
  \[
  r_\beta \leq R \leq R_\beta.
  \]

  {\it Step 3. (Decomposition of the surfaces)} Inspired by
  \eqref{inclu5} we make the following construction.  We will
  recursively construct an increasing sequence\footnote{Notice that
    the sequence itself may depend on the functions $f_1$ and $f_2$,
    but the final bound will not depend on this sequence.}  $(s_k)_{k
    \geq 1}$ with the properties
  \begin{equation}
    s_{k+1} \in [s_k+\frac12 R,s_k +R]\label{eq:claimA}
  \end{equation}
  and
  \begin{equation}
    \begin{split}
      &\| f_1 \|_{L^2(\Sigma_1 \cap \{|x_3 -s_{k+1}|\leq r_\beta\})}
      \leq 2^{-17} R^{\frac\beta{4}}\|f_1 \|_{L^2 (\Sigma_1 \cap \{\frac12 R \leq x_3-s_k \leq R\})},\\
      &\| f_2 \|_{L^2(\Sigma_2 \cap \{|x_3 -s_{k+1}|\leq r_\beta\})}
      \leq 2^{-17} R^{\frac\beta{4}} \|f_2 \|_{L^2 (\Sigma_2 \cap
        \{\frac12 R \leq x_3-s_k \leq R\})}.
    \end{split}
    \label{eq:claimB}
  \end{equation}
  In order to do so, we set $s_1 = a - r_\beta$ for a number $a$ such
  that $f_1$ and $f_2$ have support in the slab $\R^2 \times [a,b]$.
  Then \eqref{eq:claimB} is trivially verified for $k=1$.  Assume we
  already have constructed $s_k$ for some $k\geq 1$.  The set
  $\{\frac12 R \leq x_3-s_k \leq R\}\subset \R^3$ comprises $m$ slabs
  $I_1,\ldots I_m$ of thickness (in the third coordinate direction)
  $2r_\beta$, where $m$ denotes the largest integer which is smaller
  or equal to the ratio $\frac{\frac12 R}{2r_\beta}$.  For the
  function
  \begin{equation*}
    \alpha_i=\|f_i\|^{-2}_{L^2 (\Sigma_i \cap \{\frac12 R \leq x_3-s_k \leq R\})}f_i^2
  \end{equation*}
  it follows that
  \begin{align*}
    2
    &\geq \sum_{l=1}^m\left(\int_{I_l\cap \Sigma_1} \alpha_1 d\mu_1 +\int_{I_l\cap \Sigma_2} \alpha_2 d\mu_2\right)\\
    &\geq m \min_{l=1,\ldots,m} \left(\int_{I_l\cap \Sigma_1} \alpha_1
      d\mu_1 +\int_{I_l\cap \Sigma_2} \alpha_2 d\mu_2 \right).
  \end{align*}
  This estimate implies that there exists $I_{l^*}$ such that
  \begin{equation*}
    \int_{I_{l^*} \cap \Sigma_i} f_i^2 d\mu_1 \leq \frac{2}{m}\|f_i\|^2_{L^2 (\Sigma_i \cap \{\frac12 R \leq x_3-s_k \leq R\})}\; , i=1,2.
  \end{equation*}
  We choose $s_{k+1}$ to be the center of $I_{l^*}$, which
  satisfies\eqref{eq:claimA} because $I_{l^*}\subset [s_k+\frac12 R,
  s_k + R]$ and \eqref{eq:claimB} because $m^{-1}\leq 2^{-36}
  R^{\frac{\beta}{2}}$.

  For this sequence we define
  \begin{align*}
    \Sigma_{i}[k, \mathfrak{e}_3] &=\Sigma_{i} \cap \{ (x_1,x_2,x_3) \mid (-1)^i x_3 \in [ s_k+r_\beta, s_{k+1}-r_\beta] \},\\
    \tilde{\Sigma}_{i}[k, \mathfrak{e}_3]& =\Sigma_{i} \cap \{
    (x_1,x_2,x_3) \mid (-1)^i x_3 \in [s_k-r_\beta,s_k+r_\beta] \}.
  \end{align*}
  With this notation it follows that
  \begin{align}
    \label{cond1} & s_k + \frac12 R  \leq s_{k+1}  \leq s_k+  R\\
    \label{cond2} &\| f_i \|_{L^2(\tilde{\Sigma}_{i}[k,\mathfrak{e}_3
      ])} \leq 2^{-17} R^{\frac\beta4} \| f_i \|_{L^2(\Sigma_{i}[k-1,
      \mathfrak{e}_3]\cup \tilde{\Sigma}_{i}[k, \mathfrak{e}_3])\cup
      \Sigma_{i}[k, \mathfrak{e}_3])}\, , i=1,2.
  \end{align}
  Since \eqref{inclu5} has the (more restrictive) analog \eqref{inclu}
  in all directions, we can perform a similar construction to define
  $\Sigma_i[k,\mathfrak{e}_1]$ and
  $\tilde{\Sigma}_i[k,\mathfrak{e}_1]$ for $i=2,3$ as well as
  $\Sigma_i[k,\mathfrak{e}_2]$ and
  $\tilde{\Sigma}_i[k,\mathfrak{e}_2]$ for $i=1,3$ with the same
  properties \eqref{cond1} and \eqref{cond2}.  Moreover, we introduce
  \begin{align*}
    \Sigma_{1}^{k_2,k_3}={}&\Sigma_1[k_2,\mathfrak{e}_2] \cap \Sigma_1[k_3,\mathfrak{e}_3],\\
    \tilde{\Sigma}_{1}^{k_2,k_3}={}&( \tilde{\Sigma}_1[k_2,\mathfrak{e}_2] \cap \Sigma_1[k_3,\mathfrak{e}_3] ) \cup ( \Sigma_1[k_2,\mathfrak{e}_2] \cap \tilde{\Sigma}_1[k_3,\mathfrak{e}_3] )\\
    &\cup ( \tilde{\Sigma}_1[k_2,\mathfrak{e}_2] \cap
    \tilde{\Sigma}_1[k_3,\mathfrak{e}_3] ),
  \end{align*}
  and similarly $\Sigma_2^{k_1,k_3}$, $\tilde{\Sigma}_2^{k_1,k_3}$ and
  $\Sigma_3^{k_1,k_2}(\eps)$, $\tilde{\Sigma}_3^{k_1,k_2}(\eps)$.
  Now, we have the decompositions
  \begin{equation*}
    \Sigma_i=\bigcup_{k,l} \Sigma_i^{k,l}\cup \tilde{\Sigma}_i^{k,l}, \, i=1,2,
  \end{equation*}
  and the same for $\Sigma_3(\eps)$.

  {\it Step 4. (Properties of the new sets)} In this step we collect
  three useful facts about our new sets.

  a) Diameter: From \eqref{cond1} it follows that
  $\Sigma_{1}^{k_2,k_3}$, is generated as in \eqref{eq:surface_res} by
  $\Omega_1^{k_2,k_3} \subset C$, where $C$ is a cube of size $R$. In
  addition, since $\Omega_1^{k_2,k_3} \subset \Omega_1$, it follows
  that at the center $c_0$ of $C$ we have an estimate of type
  \eqref{eq:fixa}, namely
  \[
  |\nabla \phi_1(c_0)| \leq 2^{-40} (\sqrt{3} R_\beta)^\beta.
  \]
  A similar characterization holds true for
  $\tilde{\Sigma}_1^{k_2,k_3}, \Sigma_2^{k_1,k_3}$, etc.  This
  basically says that if in \eqref{sym} we replace each $\Sigma_i$ by
  $\Sigma_i^{k_l,k_m}$ or $\tilde{\Sigma}_i^{k_l,k_m}$, then the
  constant should be adjusted to $C(R_\beta,R,\eps)$.
  
  b) Orthogonality: The reason to introduce the decompositions from
  the previous step is to apply almost orthogonality arguments. More
  exactly we claim that
  \begin{equation} \label{orthog} (f_1 |_{\Sigma_{1}^{k_2,k_3}} \ast
    f_2|_{\Sigma_{2}^{k_1,k_3'}} \ast
    f_3|_{\Sigma_{3}^{k_1',k_2'}(\eps)})(0)=0
  \end{equation}
  unless $k_i=k_i'$ for $i=1,2,3$.

  Indeed, by definition of $\Sigma_1^{k_2,k_3}$ and
  $\Sigma_2^{k_1,k_3'}$ we have
  \begin{align*}
    &\supp(f_1 |_{\Sigma_{1}^{k_2,k_3}}  \ast f_2|_{\Sigma_{2}^{k_1,k_3'}}) \\
    \subset & \{(x_1,x_2,x_3)\mid x_3 \in
    [s_{k_3'}-s_{k_3+1}+2r_\beta,s_{k_3'+1}-s_{k_3}-2r_\beta] \}.
  \end{align*}
  For the left hand side of \eqref{orthog} to be different from zero
  it is necessary that
  \begin{equation*}
    [s_{k_3'}-s_{k_3+1}+2r_\beta,s_{k_3'+1}-s_{k_3}-2r_\beta]\cap
    [-r_\beta,r_\beta] \not=0.
  \end{equation*}
  due to \eqref{inclu5}, which leads to $k_3=k_3'$ because $(s_k)$ is
  strictly increasing.  In a similar manner it follows that $k_i=k_i'$
  for $i=1,2$.

  A similar argument, using the properties of $(s_k)_{k\geq 1}$,
  provides that if one allows in \eqref{orthog} one or more of the
  $\Sigma_i$ to be replaced by the corresponding set
  $\tilde{\Sigma}_i$, the convolution is zero unless $|k_i-k_i'| \leq
  1$.

  c) Smallness on $\tilde{\Sigma}$: The lack of perfect orthogonality
  in \eqref{orthog} when $\tilde{\Sigma}$'s are involved is
  compensated by the following smallness of mass on those sets
  \begin{equation} \label{small} \left(\sum_{k,l} \| f_i
      \|_{L^2(\tilde{\Sigma}_i^{k,l})}^2 \right)^\frac12 \leq 2^{-16}
    R^{\frac{\beta}4} \|f_i\|_{L^2(\Sigma_i)}
  \end{equation}
  We prove \eqref{small} for $i=1$ since the other cases are similar.
  From the definitions of the sets we have the straightforward
  estimate
  \begin{equation*}
    \sum_{k_2,k_3} \| f_1 \|_{L^2(\tilde{\Sigma}_1^{k_2,k_3})}^2 \leq \sum_{k_2} \| f_1 \|_{L^2(\tilde{\Sigma}_1[k_2,\mathfrak{e}_2])}^2 + \sum_{k_3} \| f_1 \|_{L^2(\tilde{\Sigma}_1[k_3,\mathfrak{e}_3])}^2
  \end{equation*}
  Then one uses \eqref{cond2} and the analog of \eqref{cond2} for
  $\tilde{\Sigma}_1[k_2,\mathfrak{e}_2]$ to estimate each term by
  $2^{-34} R^{\frac{\beta}{2}}\|f_1\|_{L^2(\Sigma_1)}^2$ and obtains
  \eqref{small}.

  {\it Step 5. (Conclusion of the proof)} Based on the above analysis
  on sets we decompose
  
  \begin{equation}\label{eq:decsum}
    ( f_1|_{\Sigma_{1}} \ast f_2 |_{\Sigma_{2}} \ast f_3|_{\Sigma_{3}(\eps)} )(0)
    = S+T
  \end{equation}
  where
  \begin{equation*}
    S=\sum_{k_1,k_2,k_3,k_1',k_2',k_3'}
    f_1 |_{\Sigma_{1}^{k_2,k_3}} \ast f_2|_{\Sigma_{2}^{k_1,k_3'}} \ast f_3|_{\Sigma_{3}^{k_1',k_2'}(\eps)}(0)
  \end{equation*}
  where the remainder $T$ contains $7$ sums of the same type as $S$,
  except that one ($3$ cases), two ($3$ cases) or all three ($1$ case)
  $\Sigma$ are replaced by $\tilde{\Sigma}$. We decompose as
  \begin{equation}\label{eq:remainder}
    T=T_1+T_2+T_3+T_{12}+T_{13}+T_{23}+T_{123}
  \end{equation}
  where the subscripts indicate the positions of the $\tilde{\Sigma}$
  factors.

  On behalf of the orthogonality relation \eqref{orthog} we observe
  that
  \begin{equation*}
    S= \sum_{k_1,k_2,k_3} f_1 |_{\Sigma_{1}^{k_2,k_3}} \ast f_2|_{\Sigma_{2}^{k_1,k_3}} \ast f_3|_{\Sigma_{3}^{k_1,k_2}(\eps)}(0).
  \end{equation*}

  From the conclusions in Step 4 a) above, we obtain
  \begin{align*}
    S&\leq C(R_\beta, R, \eps) \sum_{k_1,k_2,k_3} \| f_1 \|_{L^2(\Sigma_{1}^{k_2,k_3})} \| f_2 \|_{L^2(\Sigma_{2}^{k_1,k_3})}  \| f_3 \|_{L^2(\Sigma_{3}^{k_1,k_2}(\eps))} \\
    &\leq C(R_\beta, R , \eps) \| f_1 \|_{L^2(\Sigma_1)} \| f_2
    \|_{L^2(\Sigma_2)} \| f_3 \|_{L^2(\Sigma_3(\eps))}
  \end{align*}
  In passing to the last line we have used the Cauchy-Schwartz
  inequality with respect to all three $k_i$'s and the fact that the
  sets $\Sigma_1^{k_2,k_3}$ are disjoint with respect to the pair
  $(k_2,k_3)$ (and the same for the sets $\Sigma_2^{k_1,k_3}$,
  $\Sigma_3^{k_1,k_2}$).

  For each term in the remainder the same argument applies up to the
  orthogonality issue.  By the almost orthogonality the first term in
  the remainder is given by
  \begin{equation} \label{generic} T_1 =\sum_\ast \left(f_1
      |_{\tilde{\Sigma}_{1}^{k_2,k_3}} \ast
      f_2|_{\Sigma_{2}^{k_1,k_3'}} \ast
      f_3|_{\Sigma_{3}^{k_1',k_2'}(\eps)} \right)(0)
  \end{equation}
  where $\ast$ indicates summation with respect to
  $k_1,k_2,k_3,k_1',k_2',k_3'$ satisfying $|k_i-k_i'| \leq 1$ for
  $i=1,2,3$.  We obtain
  \begin{align*}
    T_1&\leq C(R_\beta, R, \eps) \sum_{\ast} \| f_1 \|_{L^2(\tilde{\Sigma}_{1}^{k_2,k_3})} \| f_2 \|_{L^2(\Sigma_{2}^{k_1,k_3'})}  \| f_3 \|_{L^2(\Sigma_{3}^{k_1',k_2'}(\eps))} \\
    &\leq 27 C(R_\beta, R, \eps) \left( \sum_{k_2,k_3} \| f_1 \|^2_{L^2(\tilde{\Sigma}_1^{k_2,k_3})}  \right)^\frac12 \| f_2 \|_{L^2(\Sigma_2)} \| f_3 \|_{L^2(\Sigma_3(\eps))} \\
    &\leq 2^{-11} R^\frac\beta4 C(R_\beta, R, \eps) \| f_1
    \|_{L^2(\Sigma_1)} \| f_2 \|_{L^2(\Sigma_2)} \| f_3
    \|_{L^2(\Sigma_3(\eps))}.
  \end{align*}
  \noindent
  where we have used \eqref{small} in passing to the last line.  If
  one considers the remaining terms in \eqref{eq:remainder} the same
  estimate holds true, which amounts to
  \begin{equation*}
    T \leq  R^\frac\beta4  C(R_\beta, R, \eps) \| f_1 \|_{L^2(\Sigma_1)} \| f_2 \|_{L^2(\Sigma_2)} \| f_3 \|_{L^2(\Sigma_3(\eps))}
  \end{equation*} 
  This estimate for the remainder term $T$ together with the estimate
  for the main term $S$ and \eqref{eq:decsum} leads to
  \eqref{indest1}.
\end{proof}

As a corollary we obtain
\begin{cor}
  Under the assumption $(2^{40}\eps)^{\frac2{\beta+2}} \leq R$ the
  estimate
  \begin{equation} \label{indest} C(R_\beta,\eps) \leq (1+
    R^{\frac\beta4}_\beta )^2 C((1+R_\beta^\beta)^{1+\frac1\beta}R,
    \eps (1+R^\beta_\beta)^{1+\frac1\beta})
  \end{equation}
  holds true, where $R_\beta$ is defined as
  $R_\beta=R^{\frac{\beta+2}{2(\beta+1)}}$.
\end{cor}

\begin{proof} This result is a direct consequence of \eqref{indest1}
  and \eqref{general}.
\end{proof}

We can now proceed with the proof of the result claimed in
\eqref{eq:eps_thick}.
\begin{proof}[Proof of Proposition \ref{prop:eps_thick}]
  We recursiveley define $R(k)$ as
  \[
  R(k+1)=
  R(k)^{\frac{2(\beta+1)}{\beta+2}}(1+R(k)^\beta)^{1+\frac1\beta}.
  \]
  A straightforward computation gives
  \begin{equation} \label{ratio} \frac{R(k+1)}{R(k)}=
    R(k)^{\frac{\beta}{\beta+2}}(1+R(k)^\beta)^{1+\frac1\beta}.
  \end{equation}
  Since the right-hand side is an increasing function in $R(k)$, one
  can choose $R(0)=C_\beta$ such that additionally
  \[
  R(0)^{\frac{\beta}{\beta+2}}(1+R(0)^\beta)^{1+\frac1\beta} \leq
  \frac12
  \]
  is satisfied.  With this choice the sequence $R(k)$ is strictly
  decreasing and $\lim_{k \rightarrow \infty} R(k)\leq \lim_{k\to
    \infty}2^{-k}=0$.  The result in \eqref{indest} provides
  \begin{equation}\label{rec}
    C(R(k),\eps) \leq  (1+R(k)^\frac\beta4)^2  C(R (k+1), \eps(1+R(k)^\beta)^{1+\frac1\beta}).
  \end{equation}
  However, in order to apply the above inequality we need to verify
  the required bounds. If
  \begin{equation*}
    \eps(k+1) =\eps \prod_{l=0}^{k-1} (1+ R(l)^\beta)^{1+\frac1\beta},
  \end{equation*}
  then the above formula can be used as long as $
  (2^{40}\eps(k))^{\frac2{\beta+2}} \leq R(k)$.  For $k=0$ one needs
  to verify $(2^{40}\eps)^{\frac2{\beta+2}} \leq C_\beta$ for which it
  is enough to take $\eps \leq
  \eps_\beta=2^{-40}C_\beta^{\frac{\beta+2}2}$.  For $k\geq 1$ we have
  that
  \[
  \eps(k)^{\frac2{\beta+2}} = \eps^{\frac2{\beta+2}}
  (\prod_{l=0}^{k-1} (1+R(l)^\beta))^{\frac2{\beta+2}(1+\frac1\beta)}
  \]
  is an increasing sequence, therefore we can find $N$ to be the
  highest value of $k$ with the property $ \eps(k)^{\frac2{\beta+2}}
  \leq R(k)$.

  Now we can use \eqref{rec} for all $k \leq N$ and by iterating it,
  we obtain
  \begin{equation} \label{ineq} C(R(0), \eps) \leq \prod_{k=0}^{N-1}
    (1+ R(k)^\frac\beta4)^2 C(R (N),\eps(N) ).
  \end{equation}
  Using \eqref{ratio} we estimate
  \begin{align*}
    \ln \prod_{k=0}^\infty (1+ R(k)^\beta)^{1+\frac1\beta} &\leq
    (1+\frac1\beta) \sum_{k=0}^\infty R(k)^\beta \leq (1+\frac1\beta)
    R(0)^\beta \sum_{k=0}^\infty (\frac{1}{2^\beta})^{k},
  \end{align*}
  which is less than $\ln(2)$ by making $C_\beta$ small enough, which
  shows that $\prod_{k=0}^\infty (1+ R(k)^\beta)^{1+\frac1\beta} \leq
  2$ and therefore $\eps(k) \leq 2 \eps$ for all $k \geq 0$.

  Since $R(N+1) \leq (2^{40}\eps(N+1))^{\frac2{\beta+2}} \leq
  (2^{41}\eps(N))^{\frac2{\beta+2}}$, it follows that $R(N) \leq
  2^{41} (\eps(N))^\frac{1}{\beta+1} \leq (2^{41}
  \eps(N))^\frac{1}{\beta+1}$.  Now we can apply the result in Lemma
  \ref{lem:start} and obtain
  \begin{equation} \label{N} C(R (N),\eps(N) ) \leq C(R (N), 2^{41}
    \eps(N) ) \leq\sqrt{2^{41}\eps(N)} \leq 2^{21} \sqrt{\eps}.
  \end{equation}
  In a similar manner as above we obtain
  $$
  \ln \prod_{k=0}^\infty (1+R(k)^\frac\beta4)^2 \leq 1
  $$
  at the expense of adjusting $C_\beta$ again. The last estimate
  together with \eqref{ineq} and \eqref{N} proves that
  \[
  C(R(0), \eps) \leq C \sqrt{\eps}
  \]
  for all $\eps \leq \eps_\beta$. From this the claim in Proposition
  \ref{prop:eps_thick} follows by partitioning each $\Sigma_i$ into a
  finite number of pieces of diameter less than $R(0)$ and, in
  addition, by partitioning $\Sigma_3(\eps)$ into a finite number of
  pieces $\Sigma_3'(\eps_\beta)$ where $\Sigma_3'$ are translates of
  $\Sigma_3$ in the $z$-direction.
\end{proof}

\subsection*{Acknowledgements}
The authors are grateful to Terence Tao for pointing out the references \cite{MR2275834,MR2155223},
and to Anthony Carbery for his helpful comments on an earlier version of this paper.

I.B. was supported in part by NSF grant DMS0738442, S.H. acknowledges support from NSF grant
DMS0354539 and D.T. was supported in part by NSF grant DMS0801261.
\bibliographystyle{plain}

\end{document}